\documentclass[12pt,twoside]{article}
\usepackage{verbatim,amssymb,amsmath}
\usepackage{graphicx,array}
\usepackage{mystyle}

\usepackage[toc]{appendix}

\usepackage{xcolor}
\usepackage{relsize}

\begin{document}

\title{Non-uniform dependence on periodic initial data for the two-component Fornberg-Whitham system in Besov spaces}
\author{
{Prerona Dutta} 
{\thanks{Affiliation during this project - Department of Mathematics, The Ohio State University
}\thanks{Current affiliation - Department of Mathematics, Xavier University of Louisiana : pdutta@xula.edu}}{, Barbara Lee Keyfitz}
\thanks{Affiliation - Department of Mathematics, The Ohio State University : keyfitz.2@osu.edu
}
}
\markboth{Dutta, Keyfitz}{Fornberg-Whitham System in Besov spaces}

\maketitle

\begin{abstract}
This paper establishes non-uniform continuity of the data-to-solution map in the periodic case, for the two-component Fornberg-Whitham system in Besov spaces $B^s_{p,r}(\mathbb{T}) \times B^{s-1}_{p,r}(\mathbb{T})$ for $s> \max\{2+\frac{1}{p}, \frac{5}{2}\}$. In particular, when $p=2$ and $r=2$, this proves the non-uniform dependence on initial data for the system in Sobolev spaces $H^s(\mathbb{T})\times H^{s-1}(\mathbb{T})$ for $s> \frac{5}{2}$.\\


\end{abstract}

\begin{center}  
{\em \today}\end{center}

\section{Introduction}\label{sec1}

In this paper, we consider the following two-component Fornberg-Whitham (FW) system for a fluid
\begin{equation}\label{FW0}
\begin{cases}
u_t + uu_x = \left(1-\partial_x^2\right)^{-1}\partial_x\left(\rho - u\right)\\
\rho_t + \left(\rho u\right)_x = 0\\
\left(u, \rho\right)(0, x) = \left(u_0, \rho_0\right)(x)
\end{cases}
\end{equation}
where $x\in\mathbb{T} = \mathbb{R}/2\pi\mathbb{Z}$, $t\in\mathbb{R}^{+}$. Here, $u=u(x,t)$ is the horizontal velocity of the fluid and $\rho = \rho(x,t)$ is the height of the fluid surface above a horizontal bottom. This system was first proposed in \cite{FTYY}, and local well-posedness and non-uniform dependence on the initial data were established in Sobolev spaces $H^s(\mathbb{R}) \times H^{s-1}(\mathbb{R})$ for $s>\frac{3}{2}$ in \cite{XZL, YL}.
\medskip

\noindent Local well-posedness in Besov spaces $B^s_{p,r}(\mathbb{R}) \times B^{s-1}_{p,r}(\mathbb{R})$ of \eqref{FW0} was established in \cite{D} for $s > \max\{2+\frac{1}{p}, \frac{5}{2}\}$. Besov spaces $B^s_{p,r}$ are a class of functions of interest in the study of nonlinear partial differential equations, as they are based on Sobolev spaces and introduce a measure of generalized H\"{o}lder regularity through the index $r$, along with the Sobolev index of differentiability $s$ and the Lebesgue integrability index $p$. If $s$ and $p$ are fixed, the spaces $B^s_{p,r}$ grow larger with increasing $r$. In \cite{D}, the FW system was shown to be well-posed in the sense of Hadamard, by establishing existence and uniqueness of the solution to the system \eqref{FW0} and then proving continuity of the data-to-solution map when the initial data belong to $B^s_{p,r}(\mathbb{R}) \times B^{s-1}_{p,r}(\mathbb{R})$ for $s > \max\{2+\frac{1}{p}, \frac{5}{2}\}$.
\medskip

\noindent In this paper, our objective is to prove nonuniform dependence on periodic initial data for the two-component FW system \eqref{FW0} in $B^s_{p,r}(\mathbb{T}) \times B^{s-1}_{p,r}(\mathbb{T})$ for $s> \max\{2+\frac{1}{p}, \frac{5}{2}\}$. We work with periodic initial data as that simplifies our choice of approximate solutions and the resulting estimates. Setting $\Lambda = 1-\partial_x^2$, we rewrite \eqref{FW0} as
\begin{equation}\label{FW}
\begin{cases}
u_t + uu_x = \Lambda^{-1}\partial_x\left(\rho - u\right)\\
\rho_t + u\rho_x + \rho u_x = 0\\
\left(u, \rho\right)(0, x) = \left(u_0, \rho_0\right)(x)
\end{cases}
\end{equation}
where $x\in\mathbb{T} = \mathbb{R}/2\pi\mathbb{Z}$ and $t\in\mathbb{R}^{+}$.
\medskip

\noindent The paper is organized as follows. In Section \ref{sec2}, we recall standard definitions and properties of Besov spaces, linear transport equations, the operator $\Lambda$ and the two-component FW system. In Section \ref{sec3}, we prove non-uniform dependence on initial data for the FW system \eqref{FW} when the initial data belong to $B^s_{p,r}(\mathbb{T}) \times B^{s-1}_{p,r}(\mathbb{T})$ for $s> \max\{2+\frac{1}{p}, \frac{5}{2}\}$. For this proof, we use a technique previously seen in the study of non-uniform continuity of data-to-solution maps for other nonlinear PDEs, for instance in \cite{HM, HT, HTi, KT, YL}. We construct two sequences of approximate solutions such that the initial data for these sequences converge to each other in $B^s_{p,r}(\mathbb{T}) \times B^{s-1}_{p,r}(\mathbb{T})$. Non-uniform convergence is then established by proving that the approximate and hence the exact solutions remain bounded away from each other for any positive time $t>0$. This idea was first explored by Kato in \cite{K} to show that the data-to-solution map for Burgers' equation is not H\"{o}lder continuous in the $H^s$ norm with $s > 3/2$ for any H\"{o}lder exponent.

\section{Notation and Preliminaries}\label{sec2}

This section is a review of relevant definitions and results on Besov spaces, linear transport equations, the operator $\Lambda$ and the two-component FW system \eqref{FW}. We begin by listing some useful notation to be used throughout Section 3.

\subsection{Notation}
For any $x,y \in \R$,
\begin{itemize}
\item $x\lesssim y$ denotes $x\leq \alpha y$ for some constant $\alpha$.
\item $x\approx y$ denotes $x = \beta y$ for some constant $\beta$.
\item $x\gtrsim y$ denotes $x\geq \gamma y$ for some constant $\gamma$.
\end{itemize}
\subsection{Besov spaces}

We recall the construction of a dyadic partition of unity from \cite{HTi}. Consider a smooth bump function $\chi$ such that $\mathrm{supp}~ \chi = [-\frac{4}{3}, \frac{4}{3}]$ and $\chi = 1$ on $[-\frac{3}{4}, \frac{3}{4}]$. For $\xi >0$, set $\varphi_{-1}(\xi) = \chi$, $\varphi_0(\xi) =  \chi\left(\frac{\xi}{2}\right) - \chi(\xi)$ and $\varphi_q(\xi) = \varphi_0(2^{-q}\xi)$. Then, $\mathrm{supp}~\varphi_q = [\frac{3}{4}\cdot 2^q, \frac{8}{3}\cdot 2^q]$ and $\sum\limits_{q\geq -1}\varphi_q(\xi) = 1$. Using this partition, a Littlewood-Paley decomposition of any periodic distribution $u$ is defined in \cite{D2} as follows.

\begin{definition}[Littlewood-Paley decomposition]\label{defnbesov}
For any $u\in\mathcal{D}'(\mathbb{T})$ with Fourier series $u(x) = \sum\limits_{j\in\mathbb{Z}} \hat{u}_j e^{ijx}$ where $\hat{u}_j = \frac{1}{2\pi}\int\limits_0^{2\pi}e^{-ijy}u(y)~dy$, its Littlewood-Paley decomposition is given by $u = \sum\limits_{q\geq -1} \Delta_q u$, where $\Delta_q u$ are periodic dyadic blocks defined for all $q\in\mathbb{Z}$ as $$\Delta_q u = \sum\limits_{j\in\mathbb{Z}}\varphi_q(j)\hat{u}_j e^{ijx}~.$$
\end{definition}

\noindent Using this Littlewood-Paley decomposition, Besov spaces on $\mathbb{T}$ are defined in \cite{D2} as follows.
\begin{definition}[Besov spaces]\label{besovdefn}
Let $s\in \R$ and $p$, $r$ $\in [1, \infty]$. Then the Besov spaces of functions are defined as
$$B^s_{p,r} \equiv B^s_{p,r} (\mathbb{T}) = \{u \in \mathcal{D}'(\mathbb{T})~\big|~ \|u\|_{B^s_{p,r} } < \infty\}~,$$
where
$$
\|u\|_{B^s_{p,r} } = 
\begin{cases}
\left(\sum\limits_{q\geq -1}(2^{sq}\|\Delta_q u\|_{{\bf L}^p})^r\right)^{1/r}~~ \text{if}~ 1\leq r < \infty\\ 
\sup\limits_{q\geq -1} 2^{sq}\|\Delta_q u\|_{{\bf L}^p}~~~~~~~~~~~~~ \text{if}~ r = \infty
\end{cases}
~.
$$
\end{definition}
\noindent Following are some properties proved in \cite[Section 2.8]{BCD} and \cite[Section 1.3]{D2} that facilitate the study of nonlinear partial differential equations in Besov spaces.
\begin{lemma}\label{besov}
Let $s, s_j \in \R$ for $j=1,2$ and $1\leq p, r \leq \infty$. Then the following hold:
\begin{enumerate}
\item[(1)] Topological property: $B^s_{p,r}$ is a Banach space which is continuously embedded in $\mathcal{D}'(\mathbb{T})$.
\item[(2)] Algebraic property: For all $s>0$, $B^s_{p,r} \cap {\bf L}^{\infty}$ is a Banach algebra.
\item[(3)] Interpolation: If $f \in B^{s_1}_{p,r} \cap B^{s_2}_{p,r}$ and $\theta \in [0,1]$, then $f \in B^{\theta s_1 + (1-\theta) s_2}_{p,r}$ and 
$$\|f\|_{B^{\theta s_1 + (1-\theta) s_2}_{p,r}} \leq  \|f\|^{\theta}_{B^{s_1}_{p,r}} \|f\|^{1-\theta}_{B^{s_2}_{p,r}} ~.$$
\item[(4)] Embedding: $B^{s_1}_{p,r} \hookrightarrow B^{s_2}_{p,r}$ whenever $s_1 \geq s_2$. In particular, $B^{s}_{p,r} \hookrightarrow B^{s-1}_{p,r}$ for all $s \in \R$.
\end{enumerate}
\end{lemma}
{\bf Remark on (2) in Lemma \ref{besov}:} When $s> \frac{1}{p}$ (or $s\geq \frac{1}{p}$ and $r=1$), $B^s_{p,r} \hookrightarrow {\bf L}^{\infty}$. We will use the fact that for $0<s<\frac{1}{p}$, the result is still true as long as the functions are bounded. 

\subsection{Linear transport equation}

Given a linear transport equation, Proposition A.1 in \cite{D1} proves the following estimate for its solution size in Besov spaces.
\begin{proposition}\label{danchinLT}
Consider the linear transport equation
\begin{equation}\label{lineartrans}
\begin{cases}
\partial_t f + v \partial_x f = F \\
f(x,0) = f_0(x)
\end{cases}
\end{equation}
where $f_0 \in B^s_{p,r}(\mathbb{T})$, $F \in {\bf L}^1((0, T); B^{s}_{p,r}(\mathbb{T}))$ and $v$ is such that $\partial_x v \in {\bf L}^1((0, T); B^{s-1}_{p,r}(\mathbb{T}))$. Suppose $f \in {\bf L}^{\infty}((0, T); B^{s}_{p,r}(\mathbb{T})) \cap C([0,T]; \mathcal{D}'(\mathbb{T}))$ is a solution to \eqref{lineartrans}. Let $1\leq p, r \leq \infty$. If either $s\neq 1+\frac{1}{p}$, or $s= 1+\frac{1}{p}$ and $r=1$, then for a positive constant $C$ which depends on $s$, $p$ and $r$, we have
\begin{equation*}\label{transestimate}
\|f(t)\|_{B^s_{p,r}}~\leq~ e^{CV(t)}\left(\|f_0\|_{B^s_{p,r}} + C\int_0^t  e^{-CV(\tau)}\|F(\tau)\|_{B^s_{p,r}}~d\tau\right)
\end{equation*}
where
$$V(t) = \int_0^t \|\partial_x v(\tau)\|_{B^{1/p}_{p,r}\cap {\bf L}^{\infty}}~d\tau~~~~ \textrm{if}~s < 1+\frac{1}{p}$$
and
$$V(t) = \int_0^t \|\partial_x v(\tau)\|_{B^{s-1}_{p,r}}~d\tau~~~~ \textrm{otherwise}~.$$
For $r<\infty$, $f\in C([0,T], B^{s}_{p,r}(\mathbb{T}))$ and if $r=\infty$, then $f\in C([0,T], B^{s'}_{p,1}(\mathbb{T}))$ for all $s'<s$.
\end{proposition}

\subsection{The Operator $\Lambda$}
Let $\Lambda = 1-\partial_x^2$; then for any test function $g$, the Fourier transform of $\Lambda^{-1}g$ is given by $\mathcal{F}\left(\Lambda^{-1}g\right) = \frac{1}{1+\xi^2}\hat{g}(\xi)$. Moreover, for any $s\in\R$, $\Lambda^{-1}\partial_x$ is continuous from $B^{s-1}_{p,r}$ to $B^{s}_{p,r}$; that is, for all $h\in B^{s-1}_{p,r}$, there exists a constant $\kappa > 0$ depending on $s,p$ and $r$ such that
\begin{equation}\label{contlamb}
\|\Lambda^{-1}\partial_xh\|_{B^{s}_{p,r}}~\leq~\kappa\|h\|_{B^{s-1}_{p,r}}~.
\end{equation}
 
\subsection{The Fornberg-Whitham system}
The well-posedness of the two-component FW system \eqref{FW} in Besov spaces was established on the real line in \cite{D} with the following result.
\begin{theorem}\label{existmain}
Let $s> \max\{2+\frac{1}{p}, \frac{5}{2}\}$, $p \in [1, \infty]$, $r \in [1, \infty)$ and $(u_0, \rho_0) \in B^s_{p,r}(\mathbb{R})  \times B^{s-1}_{p,r}(\mathbb{R}) $. Then the system \eqref{FW} has a unique solution $(u, \rho) \in \\C\left([0,T];B^s_{p,r}(\mathbb{R}) \times B^{s-1}_{p,r}(\mathbb{R}) \right)$ where the lifespan $T$ is given by
$$T~=~  \frac{C}{\left(\|u_0\|_{B^s_{p,r}} + \|\rho_0\|_{B^{s-1}_{p,r}}\right)^2}~,$$
with $C$ being a constant that depends on $s$, $p$ and $r$, and the solution size is estimated as 
\begin{equation*}\label{solnsize}
\|(u, \rho)\|_{B^s_{p,r}  \times B^{s-1}_{p,r}} ~\leq~ 2\left(\|u_0\|_{B^s_{p,r}} + \|\rho_0\|_{B^{s-1}_{p,r}}\right)~.
\end{equation*}
Moreover, the data-to-solution map is continuous.
\end{theorem}
Since we work with $B^s_{p,r}(\mathbb{T})  \times B^{s-1}_{p,r}(\mathbb{T})$ in this paper, we state the following.
\begin{corollary}\label{corexist}
Theorem \ref{existmain} holds when $\mathbb{R}$ is replaced by $\mathbb{T}$.
\end{corollary}
\begin{proof} 
Existence of a solution to \eqref{FW} is proved by altering the mollifier used to prove Theorem \ref{existmain}. This adaptation of the mollifier was done for the single Fornberg-Whitham equation in \cite[Section 3.1]{HT}. Uniqueness and continuous dependence on periodic initial data for the system \eqref{FW} are established by approximation arguments similar to those in \cite[Sections 3.2-3.3]{D}.

\end{proof}

\section{Non-uniform dependence on initial data}\label{sec3}

In this section we establish nonuniform dependence on initial data in the periodic case for the two-component FW system \eqref{FW} in Besov spaces. 

\begin{theorem}\label{mainthm}
Let $s> \max\{2+\frac{1}{p}, \frac{5}{2}\}$ and $r \in [1, \infty]$. The data-to-solution map $(u_0, \rho_0) \mapsto (u(t), \rho(t))$ of the Cauchy problem \eqref{FW} is not uniformly continuous from any bounded subset of $B^s_{p,r}(\mathbb{T}) \times B^{s-1}_{p,r}(\mathbb{T})$ into $\mathcal{C}([0,T];B^s_{p,r}(\mathbb{T}))\times\mathcal{C}([0,T];B^{s-1}_{p,r}(\mathbb{T}))$ where $T$ is given by Theorem \ref{existmain}. In particular, there exist two sequences of solutions $\{(u_{\omega,n}, \rho_{\omega,n})\}$ with $\omega = \pm 1$ such that the following hold
\begin{enumerate}
\item[(i)] $\lim\limits_{n\to\infty} \left(\|u_{1,n}(0) - u_{-1,n}(0)\|_{B^s_{p,r}} +\|\rho_{1,n}(0) - \rho_{-1,n}(0)\|_{B^{s-1}_{p,r}}\right) = 0$.
\item[(ii)] $\liminf\limits_{n\to\infty}\left(\|u_{1,n} - u_{-1,n}\|_{B^s_{p,r}} + \|\rho_{1,n} - \rho_{-1,n}\|_{B^{s-1}_{p,r}}\right)\gtrsim |\sin t|$.
\end{enumerate}
\end{theorem}

\begin{proof}
For $n\in\mathbb{N}$, we consider two sequences of functions $\{(u^{\omega,n}, \rho^{\omega,n})\}$ with $\omega = \pm 1$, defined by
\begin{equation*}
\begin{cases}
u^{\omega,n} = \frac{-\omega}{n} + \frac{1}{n^s} \sin(nx+\omega t)\\
\rho^{\omega,n} = \frac{1}{n}+\frac{1}{n^{s}} \sin(nx+\omega t)
\end{cases}~.
\end{equation*}
We take initial data 
\begin{equation*}\label{initialdata}
\begin{cases}
u^0_{\omega,n} = u^{\omega,n}(0) = \frac{-\omega}{n} + \frac{1}{n^s} \sin nx\\
\rho^0_{\omega,n} = \rho^{\omega,n}(0) = \frac{1}{n} + \frac{1}{n^{s}} \sin nx
\end{cases}~.
\end{equation*}
Let the solutions to the FW system \eqref{FW} with these initial data be denoted by $(u_{\omega,n}, \rho_{\omega,n})$. At $t=0$, we have
 \begin{equation*}\label{zeroinitial}
\lim_{n\to\infty} \left(\|u^0_{1,n} - u^0_{-1,n}\|_{B^s_{p,r}} +\|\rho^0_{1,n} - \rho^0_{-1,n}\|_{B^{s-1}_{p,r}}\right) = \lim_{n\to\infty} 2 \|n^{-1}\|_{B^s_{p,r}} = 0~,
\end{equation*}
which proves part {\it(i)} of Theorem \ref{mainthm}.\\

\noindent To prove part {\it(ii)}, first we estimate $\|(u^0_{\omega,n}, \rho^0_{\omega,n})\|_{B^{\gamma}_{p,r} \times B^{\gamma-1}_{p,r}}$ and $\|(u^{\omega,n}, \rho^{\omega,n})\|_{B^{\gamma}_{p,r} \times B^{\gamma-1}_{p,r}}$ for any $\gamma > 0$ and $r<\infty$. Using the triangle inequality,  we have
\begin{equation}\label{norminitial}
\|(u^0_{\omega,n}, \rho^0_{\omega,n})\|_{B^{\gamma}_{p,r} \times B^{\gamma-1}_{p,r}}~\leq~ 2\|n^{-1}\|_{B^{\gamma}_{p,r}} + n^{-s}\|\sin nx\|_{B^{\gamma}_{p,r}} + n^{1-s}\|\sin nx\|_{B^{\gamma-1}_{p,r}}
 \end{equation}
By Definition \ref{besovdefn},
\begin{equation}\label{sinbesov}
\|\sin nx\|_{B^{\gamma}_{p,r}} = \left(\sum\limits_{q\geq -1} 2^{\gamma q r}\|\Delta_q \sin nx\|^r_{{\bf L}^p}\right)^{\frac{1}{r}}~.
\end{equation}
From Definition \ref{defnbesov}, we have $\|\Delta_q \sin(nx)\|_{{\bf L}^p} = \varphi_q(n)$ where $0 < \varphi_q(n) \leq 1$ for all $q$ such that $\frac{1}{\ln (2)}\ln \left(\frac{3}{8}n\right)\leq q \leq \frac{1}{\ln (2)}\ln \left(\frac{4}{3}n\right)$ and $\varphi_q\left(n\right)= 0$ otherwise.
Hence \eqref{sinbesov} implies that for any $\gamma>0$,
\begin{equation*}
\|\sin (nx)\|_{B^{\gamma}_{p,r}}~\leq~ \left(\sum\limits_{q= \frac{1}{\ln (2)}\ln \left(\frac{3}{8}n\right)}^{ \frac{1}{\ln (2)}\ln \left(\frac{4}{3}n\right)} \left(2^q\right)^{\gamma r}\right)^{\frac{1}{r}}~.
\end{equation*}
As $2^q \leq \frac{4}{3}n$ for every term in the summation, from the above we get that
\begin{align}\label{sinnorm}
\|\sin (nx)\|_{B^{\gamma}_{p,r}}~&\leq~ \left(\sum\limits_{q= \frac{1}{\ln (2)}\ln \left(\frac{3}{8}n\right)}^{ \frac{1}{\ln (2)}\ln \left(\frac{4}{3}n\right)} \left(\frac{4}{3}n\right)^{\gamma r}\right)^{\frac{1}{r}} \nonumber \\
&=~ \left(\frac{1}{\ln (2)}\left[\ln \left(\frac{4}{3}n\right) - \ln \left(\frac{3}{8}n\right)\right]\right)^{\frac{1}{r}}\left(\frac{4}{3}n\right)^{\gamma} \nonumber \\
&=~ \left(\frac{1}{\ln (2)} \ln \left(\frac{32}{9}\right)\right)^{\frac{1}{r}}\left(\frac{4}{3}\right)^{\gamma}n^{\gamma}~=~ C_{\gamma} n^{\gamma}~.
\end{align}
Here and in what follows, $C_{\gamma}$ is a generic constant which depends only on $\gamma$, for fixed $p$ and $r$. Similarly, it follows that for any $\gamma>0$,
\begin{equation}\label{cosupperbd}
\|\cos (nx)\|_{B^{\gamma}_{p,r}}~\leq~C_{\gamma} n^{\gamma}.
\end{equation}
Using \eqref{sinnorm} and observing that $\Delta_q n^{-1} = 0$ for all $q>-1$, from \eqref{norminitial} we obtain
\begin{align}\label{initialnorm}
\|(u^0_{\omega,n}, \rho^0_{\omega,n})\|_{B^{\gamma}_{p,r} \times B^{\gamma-1}_{p,r}}~&\leq~ 2^{1-\gamma}n^{-1} + C_{\gamma} n^{\gamma}n^{-s}+C_{\gamma} n^{\gamma-1}n^{1-s}\nonumber \\
&\leq~ C_{\gamma}\max\{n^{-1},n^{\gamma-s}\}~.
\end{align}
Since $(u^{\omega,n},\rho^{\omega,n})$ is a phase shift of $(u^0_{\omega,n}, \rho^0_{\omega,n})$, we have 
\begin{equation}\label{approxnorm}
\|(u^{\omega,n},\rho^{\omega,n})\|_{B^{\gamma}_{p,r}\times B^{\gamma-1}_{p,r}}~\leq~C_{\gamma}\max\{n^{-1},n^{\gamma-s}\}~.
\end{equation}
If $r=\infty$, \eqref{initialnorm} and \eqref{approxnorm} follow immediately from Definition \ref{besovdefn}.\\

\noindent We complete the proof of Theorem \ref{mainthm} by establishing {\it(ii)} for $\{(u^{\omega,n}, \rho^{\omega,n})\}$, taking advantage of the following lemma, whose proof follows the proof of Theorem \ref{mainthm}. Lemma \ref{lemma} establishes that for each $n$ and $\omega$, $(u^{\omega,n}, \rho^{\omega,n})$ approximates $(u_{\omega,n}, \rho_{\omega,n})$ in $B^s_{p,r}(\mathbb{T}) \times B^{s-1}_{p,r}(\mathbb{T})$ uniformly on $[0,T]$ for some $T>0$.

\begin{lemma}\label{lemma}
Let $\mathcal{E} = (\mathcal{E}_1, \mathcal{E}_2)$ where $\mathcal{E}_1 = u_{\omega,n} - u^{\omega,n}$ and $\mathcal{E}_2 = \rho_{\omega,n} - \rho^{\omega,n}$. Then for all $t \in (0, T)$, $\|\mathcal{E}(t)\|_{B^s_{p,r}\times B^{s-1}_{p,r}} = \|\mathcal{E}_1(t)\|_{B^{s}_{p,r}} + \|\mathcal{E}_2(t)\|_{B^{s-1}_{p,r}} \to 0$ as $n \to \infty$.
\end{lemma}

\noindent We show that $(u_{-1,n},\rho_{-1,n})$ and $(u_{1,n},\rho_{1,n})$ stay bounded away from each other for any $t>0$. Since
\begin{equation}\label{uineq}
\|u_{1,n} - u_{-1,n}\|_{B^{s}_{p,r}}~\geq~\|u^{1,n} - u^{-1,n}\|_{B^{s}_{p,r}} - \|u^{1,n} - u_{1,n}\|_{B^{s}_{p,r}} - \|u^{-1,n} - u_{-1,n}\|_{B^{s}_{p,r}}
\end{equation}
and
\begin{equation}\label{rhoineq}
\|\rho_{1,n} - \rho_{-1,n}\|_{B^{s-1}_{p,r}}~\geq~\|\rho^{1,n} - \rho^{-1,n}\|_{B^{s-1}_{p,r}} - \|\rho^{1,n} - \rho_{1,n}\|_{B^{s-1}_{p,r}} - \|\rho^{-1,n} - \rho_{-1,n}\|_{B^{s-1}_{p,r}}~,
\end{equation}
and $\|\mathcal{E}_1(t)\|_{B^{s}_{p,r}}$ and $\|\mathcal{E}_2(t)\|_{B^{s-1}_{p,r}}$ go to $0$ by Lemma \ref{lemma}, the last two terms on the right hand side of both \eqref{uineq} and \eqref{rhoineq} approach $0$ as $n\to\infty$. Hence we obtain
\begin{multline}\label{urhofinish1}
\|u_{1,n} - u_{-1,n}\|_{B^{s}_{p,r}} + \|\rho_{1,n} - \rho_{-1,n}\|_{B^{s-1}_{p,r}} \geq~\|u^{1,n} - u^{-1,n}\|_{B^{s}_{p,r}} + \|\rho^{1,n} - \rho^{-1,n}\|_{B^{s-1}_{p,r}}\\
\geq n^{-s}\left(\|\sin(nx+t) - \sin(nx-t)\|_{B^{s}_{p,r}} + \|\sin(nx+t) - \sin(nx-t)\|_{B^{s-1}_{p,r}}\right)\\
= 2n^{-s}\left(\|\cos (nx) \|_{B^{s}_{p,r}}|\sin (t)| + \|\cos (nx) \|_{B^{s-1}_{p,r}}|\sin (t)|\right)~.
\end{multline}
By Definition \ref{besovdefn}, if $r=\infty$, we immediately have 
\begin{equation}\label{cosnorm}
\|\cos (nx)\|_{B^{s}_{p,r}} \geq C_s n^s~,
\end{equation}
where $C_s$ is a constant that depends only on $s$. For $1\leq r < \infty$, there is a similar estimate, whose proof is given in the Appendix.
Using \eqref{cosnorm}, it follows from \eqref{urhofinish1} that
\begin{multline*}
\liminf_{n\to\infty}\left(\|u_{1,n} - u_{-1,n}\|_{B^{s}_{p,r}} + \|\rho_{1,n} - \rho_{-1,n}\|_{B^{s-1}_{p,r}}\right)~\\\geq~2 C_s\left(\liminf_{n\to\infty} |\sin (t)| + \liminf_{n\to\infty} n^{-1} |\sin (t)|\right)~\approx~ |\sin(t)|~>~0~.
\end{multline*}
This proves part {\it(ii)} of Theorem \ref{mainthm} and completes the proof of non-uniform dependence on initial data for the two-component FW system \eqref{FW} in $B^s_{p,r}(\mathbb{T}) \times B^{s-1}_{p,r}(\mathbb{T})$ for $s> \max\{2+\frac{1}{p}, \frac{5}{2}\}$. 

\end{proof}

\noindent Now we prove Lemma \ref{lemma}.\\

\begin{proof} {\bf (Lemma \ref{lemma})} We show that $\|\mathcal{E}(t)\|_{B^{\gamma}_{p,r}\times B^{\gamma-1}_{p,r}} \to 0$ as $n \to \infty$ for any $\gamma$ with $s-\frac{3}{2} < \gamma < s-1$, and then interpolate between such a $\gamma$ and a value $\delta > s$. Recall that $\mathcal{E}_1 = u_{\omega,n} - u^{\omega,n}$ and $\mathcal{E}_2 = \rho_{\omega,n} - \rho^{\omega,n}$. It can be seen that $\mathcal{E}_1$ and $\mathcal{E}_2$  vanish at $t=0$ and that they satisfy the equations
\begin{equation}\label{e1e2transport}
\begin{cases}
\partial_t \mathcal{E}_1 + u^{\omega,n}\partial_x \mathcal{E}_1 = -\mathcal{E}_1\partial_x u_{\omega, n} + \Lambda^{-1}\partial_x (\mathcal{E}_2 - \mathcal{E}_1) - R_1\\
\partial_t \mathcal{E}_2 + u_{\omega,n}\partial_x \mathcal{E}_2 = -\mathcal{E}_2\partial_x u_{\omega, n} -\rho^{\omega,n}\partial_x\mathcal{E}_1 - \mathcal{E}_1\partial_x\rho^{\omega,n} - R_2
\end{cases}~.
\end{equation}
Here $R_1$ and $R_2$ are the FW system for the approximate solutions, that is,
\begin{equation*}\label{r1r2defn}
\begin{cases}
R_1 = \partial_t u^{\omega, n} + u^{\omega, n}\partial_x u^{\omega, n} - \Lambda^{-1}\partial_x(\rho^{\omega, n} - u^{\omega, n}) \\
R_2 = \partial_t \rho^{\omega, n} + \partial_x (\rho^{\omega,n}u^{\omega, n})
\end{cases}~.
\end{equation*}
\begin{itemize}
\item Estimate for $\|R_1\|_{B^{\gamma}_{p,r}}$: Using the definitions of $u^{\omega, n}$ and $\rho^{\omega,n}$, we have
\begin{multline*}
R_1 = \partial_t u^{\omega, n} + u^{\omega, n}\partial_x u^{\omega, n} - \Lambda^{-1}\partial_x(\rho^{\omega,n} - u^{\omega,n}) = \frac{1}{2n^{2s-1}}\sin\left(2(nx+\omega t)\right)~.
\end{multline*}
Then by \eqref{sinnorm}, 
\begin{equation*}\label{r1est}
\|R_1\|_{B^{\gamma}_{p,r}} \leq C_{\gamma}n^{\gamma-2s+1}~.
\end{equation*}
\item Estimate for $\|R_2\|_{B^{\gamma-1}_{p,r}}$: Using the definitions of $u^{\omega,n}$ and $\rho^{\omega,n}$,
\begin{multline*}
R_2 = \partial_t \rho^{\omega, n} + \partial_x (\rho^{\omega,n}u^{\omega, n})
= \frac{1}{n^{s}}\cos(nx+\omega t) + \frac{1}{n^{2s-1}}\sin\left(2(nx+\omega t)\right)~.
\end{multline*}
So from \eqref{sinnorm} and \eqref{cosupperbd} it follows that
\begin{equation*}\label{r2est}
\|R_2\|_{B^{\gamma-1}_{p,r}} \leq C_{\gamma} \left(n^{-s}n^{\gamma-1} + n^{1-2s}n^{\gamma-1}\right) \leq C_{\gamma} n^{\gamma-s-1}~.
\end{equation*}
\end{itemize}
Therefore,
\begin{equation}\label{r1r2est}
\|R_1(\tau)\|_{B^{\gamma}_{p,r}} + \|R_2(\tau)\|_{B^{\gamma-1}_{p,r}}~\lesssim~n^{\gamma-s-1}~.
\end{equation}
\noindent Since $\mathcal{E}_1(t)$ and $\mathcal{E}_2(t)$ satisfy the linear transport equations \eqref{e1e2transport}, to estimate the error $\|\mathcal{E}(t)\|_{B^{\gamma}_{p,r} \times B^{\gamma-1}_{p,r}}$, we apply Proposition \ref{danchinLT} to obtain
\begin{equation}\label{e1ineq}
\|\mathcal{E}_1(t)\|_{B^{\gamma}_{p,r}} \leq K_1e^{K_1V_1(t)}\int_0^t e^{- K_1V_1(\tau)}\|F_1({\tau})\|_{B^{\gamma}_{p,r}} ~d\tau
\end{equation}
and
\begin{equation}\label{e2ineq}
\|\mathcal{E}_2(t)\|_{B^{\gamma-1}_{p,r}} \leq K_2e^{K_2V_2(t)}\int_0^t e^{- K_2V_2(\tau)}\|F_2({\tau})\|_{B^{\gamma-1}_{p,r}} ~d\tau
\end{equation}
where $K_1$, $K_2$ are positive constants depending on $\gamma$ and
\begin{equation}\label{f1}
F_1(t) = -\mathcal{E}_1\partial_x u_{\omega, n} + \Lambda^{-1}\partial_x (\mathcal{E}_2 - \mathcal{E}_1) - R_1~,
\end{equation}
\begin{equation}\label{f2}
~~~~~F_2(t) = -\mathcal{E}_2\partial_x u_{\omega, n} -\rho^{\omega,n}\partial_x\mathcal{E}_1 - \mathcal{E}_1\partial_x\rho^{\omega,n} - R_2~.
\end{equation}
\begin{equation*}
\hspace*{-2.1cm}V_1(t) = \int_0^t \|\partial_x u^{\omega,n}(\tau)\|_{B^{\gamma-1}_{p,r}} ~d\tau~,
\end{equation*}
\begin{equation*}
\hspace*{+1.8cm}V_2(t) = 
\begin{cases}
\int_0^t \|\partial_x u_{\omega,n}(\tau)\|_{B^{1/p}_{p,r}\cap{\bf L}^{\infty}} ~d\tau~~~~\textrm{if}~\gamma<2+\frac{1}{p}\\

\int_0^t \|\partial_x u_{\omega,n}(\tau)\|_{B^{\gamma-2}_{p,r}} ~d\tau~~~~~~~~\textrm{otherwise}
\end{cases}~.
\end{equation*}
Since $s-\frac{3}{2} < \gamma < s-1$, we have
\begin{equation}\label{v1}
V_1(t) \lesssim ~ n^{\gamma - s}t \leq~n^{-1}t
~~\textrm{
and}
~~V_2(t) \leq C\int_0^t \|u_{\omega,n}(\tau)\|_{B^{\gamma}_{p,r}}~d\tau
\end{equation}
for some constant $C$ that depends on $\gamma$, $p$ and $r$. By Theorem \ref{existmain} and \eqref{initialnorm}, it follows that
\begin{equation}\label{v2}
V_2(t)  \leq 2C \int_0^t \|\left(u^0_{\omega,n}, \rho^0_{\omega,n}\right)\|_{B^{\gamma}_{p,r}\times B^{\gamma-1}_{p,r}}~d\tau~\lesssim~n^{-1}t~.
\end{equation}
Let $K = \min\{K_1, K_2\}$. Using \eqref{v1} and \eqref{v2}, we combine \eqref{e1ineq} and \eqref{e2ineq} to get,
\begin{equation}\label{e1e2}
\|\mathcal{E}_1(t)\|_{B^{\gamma}_{p,r}} + \|\mathcal{E}_2(t)\|_{B^{\gamma-1}_{p,r}} \lesssim \int_0^t e^{K(t-\tau)/n}\left(\|F_1({\tau})\|_{B^{\gamma}_{p,r}} + \|F_2(\tau)\|_{B^{\gamma-1}_{p,r}}\right) ~d\tau~.
\end{equation}
\begin{itemize}
\item Estimate for $\|F_1(\tau)\|_{B^{\gamma}_{p,r}}$: From \eqref{f1}, as $B^{\gamma}_{p,r}$ is a Banach algebra, we have
\begin{align}\label{f1ineq1}
\|F_1(\tau)\|_{B^{\gamma}_{p,r}} & \leq \|\mathcal{E}_1\|_{B^{\gamma}_{p,r}}\|\partial_x u_{\omega, n}\|_{B^{\gamma}_{p,r}} + \|\Lambda^{-1}\partial_x (\mathcal{E}_2 - \mathcal{E}_1)\|_{B^{\gamma}_{p,r}}+ \|R_1\|_{B^{\gamma}_{p,r}}~ \nonumber \\
& \leq \|\mathcal{E}_1\|_{B^{\gamma}_{p,r}}\|u_{\omega, n}\|_{B^{\gamma+1}_{p,r}} + \|\Lambda^{-1}\partial_x (\mathcal{E}_2 - \mathcal{E}_1)\|_{B^{\gamma}_{p,r}}+ \|R_1\|_{B^{\gamma}_{p,r}}~.
\end{align}
From \eqref{contlamb}, 
\begin{equation}\label{f1req}
\|\Lambda^{-1}\partial_x (\mathcal{E}_2 - \mathcal{E}_1)\|_{B^{\gamma}_{p,r}} \leq \kappa \|\mathcal{E}_2 - \mathcal{E}_1\|_{B^{\gamma-1}_{p,r}} \leq M\left(\|\mathcal{E}_1\|_{B^{\gamma}_{p,r}} + \|\mathcal{E}_2\|_{B^{\gamma-1}_{p,r}}\right)
\end{equation}
where $M$ is a constant depending on $\gamma, p$ and $r$. By Theorem \ref{existmain} we have
\begin{equation*}
\|u_{\omega, n}\|_{B^{\gamma+1}_{p,r}} \leq 2\|\left(u^0_{\omega,n}, \rho^0_{\omega,n}\right)\|_{B^{\gamma+1}_{p,r}\times B^{\gamma}_{p,r}}~,
\end{equation*}
so by \eqref{initialnorm}, 
$\|u_{\omega, n}\|_{B^{\gamma+1}_{p,r}} \leq 2C_{\gamma}\max\{n^{-1}, n^{\gamma+1-s}\}$.
As $\gamma > s-\frac{3}{2}$, 
\begin{equation}\label{gamma1}
\|u_{\omega, n}\|_{B^{\gamma+1}_{p,r}} \lesssim n^{\gamma+1-s}~.
\end{equation}
Using \eqref{f1req} and \eqref{gamma1}, from \eqref{f1ineq1} we obtain
\begin{equation}\label{f1ineq}
\|F_1(\tau)\|_{B^{\gamma}_{p,r}} \lesssim \left(M + n^{\gamma+1-s}\right)\|\mathcal{E}_1(\tau)\|_{B^{\gamma}_{p,r}} + M\|\mathcal{E}_2(\tau)\|_{B^{\gamma-1}_{p,r}} + \|R_1(\tau)\|_{B^{\gamma}_{p,r}}~.
\end{equation}
\item Estimate for $\|F_2(\tau)\|_{B^{\gamma-1}_{p,r}}$: We may use the algebra property, (2) of Lemma \ref{besov}, for $B^{\gamma-1}_{p,r}$ since $\gamma -1 > s-\frac{5}{2} > 0$ and the functions we are dealing with are bounded. Then, from \eqref{f2},
\begin{align}\label{f2ineq1}
\|F_2(t)\|_{B^{\gamma-1}_{p,r}} &\leq \|\mathcal{E}_2\|_{B^{\gamma-1}_{p,r}}\|\partial_x u_{\omega, n}\|_{B^{\gamma-1}_{p,r}} + \|\rho^{\omega,n}\|_{B^{\gamma-1}_{p,r}}\|\partial_x \mathcal{E}_1\|_{B^{\gamma-1}_{p,r}} \nonumber \\ 
&~~~~+ \|\partial_x\rho^{\omega,n}\|_{B^{\gamma-1}_{p,r}}\|\mathcal{E}_1\|_{B^{\gamma-1}_{p,r}} + \|R_2\|_{B^{\gamma-1}_{p,r}} \nonumber \\
&\lesssim n^{-1}\|\mathcal{E}_1\|_{B^{\gamma}_{p,r}} + \|\mathcal{E}_2\|_{B^{\gamma-1}_{p,r}}\|u_{\omega, n}\|_{B^{\gamma}_{p,r}} + \|R_2\|_{B^{\gamma-1}_{p,r}}~.
\end{align}
By Corollary \ref{corexist}, $\|u_{\omega, n}\|_{B^{\gamma}_{p,r}} \leq 2\|\left(u^0_{\omega,n}, \rho^0_{\omega,n}\right)\|_{B^{\gamma}_{p,r}\times B^{\gamma-1}_{p,r}}$,
which implies $$\|u_{\omega, n}\|_{B^{\gamma}_{p,r}} \leq 2C_{\gamma}\max\{n^{-1}, n^{\gamma-s}\}$$ by \eqref{initialnorm}.
 As $\gamma<s-1$, $\|u_{\omega, n}\|_{B^{\gamma}_{p,r}} \lesssim n^{-1}$. Using this in \eqref{f2ineq1} yields
\begin{equation}\label{f2ineq}
\|F_2(\tau)\|_{B^{\gamma-1}_{p,r}} \lesssim n^{-1}\|\mathcal{E}_1(\tau)\|_{B^{\gamma}_{p,r}} + n^{-1}\|\mathcal{E}_2(\tau)\|_{B^{\gamma-1}_{p,r}} + \|R_2(\tau)\|_{B^{\gamma-1}_{p,r}}~.
\end{equation}
\end{itemize}
Adding \eqref{f1ineq} and \eqref{f2ineq} gives
\begin{align}\label{f1f2}
\|F_1(\tau)\|_{B^{\gamma}_{p,r}} + \|F_2(\tau)\|_{B^{\gamma-1}_{p,r}}~& \lesssim~(M+n^{\gamma+1-s})\left(\|\mathcal{E}_1(\tau)\|_{B^{\gamma}_{p,r}} + \|\mathcal{E}_2(\tau)\|_{B^{\gamma-1}_{p,r}}\right) \nonumber \\ &~~~~~+ \|R_1(\tau)\|_{B^{\gamma}_{p,r}} + \|R_2(\tau)\|_{B^{\gamma-1}_{p,r}}~.
\end{align}
Substituting \eqref{f1f2} in \eqref{e1e2}, we obtain
\begin{equation}\label{efcombo1}
\|\mathcal{E}(t)\|_{B^{\gamma}_{p,r} \times B^{\gamma-1}_{p,r}} ~\lesssim~ f(t) + \int_0^t g(\tau)\|\mathcal{E}(\tau)\|_{B^{\gamma}_{p,r} \times B^{\gamma-1}_{p,r}}~d\tau
\end{equation}
where
\begin{equation}\label{part1}
f(t) \approx \int_0^t e^{K(t-\tau)/n}\left(\|R_1(\tau)\|_{B^{\gamma}_{p,r}} + \|R_2(\tau)\|_{B^{\gamma-1}_{p,r}}\right)~d\tau
\end{equation}
and
\begin{equation}\label{part2}
g(\tau) \approx (M+n^{\gamma+1-s})e^{K(t-\tau)/n}~\leq~(M+1)e^{K(t-\tau)/n}~.
\end{equation}
Using Gr\"{o}nwall's inequality, from \eqref{efcombo1} we obtain
\begin{equation}\label{efcombo2}
\|\mathcal{E}(t)\|_{B^{\gamma}_{p,r} \times B^{\gamma-1}_{p,r}}~\lesssim~f(t) + \int_0^t g(\tau)f(\tau)e^{\int_{\tau}^t g(z)~dz}~d\tau~.
\end{equation}
Using \eqref{r1r2est} along with \eqref{part1} and \eqref{part2}, from \eqref{efcombo2} we get \begin{equation}\label{s1}
\|\mathcal{E}(t)\|_{B^{\gamma}_{p,r}\times B^{\gamma-1}_{p,r}} \lesssim n^{\gamma-s-1}~,
\end{equation}
which means that $\|\mathcal{E}(t)\|_{B^{\gamma}_{p,r} \times B^{\gamma-1}_{p,r}} \to 0$ as $n\to\infty$ for any $s-\frac{3}{2} < \gamma < s-1$.\\

\noindent On the other hand, if $\delta \in (s, s+1)$, then noting that the solution with the given data is in $B^{\delta}_{p,r} \times B^{\delta-1}_{p,r}$ for any $\delta$ we have, for $0<t<T$ (from Theorem \ref{existmain})
\begin{align}\label{s2-1}
\|\mathcal{E}(t)\|_{B^{\delta}_{p,r}\times B^{\delta-1}_{p,r}}~ &~\leq~\|(u_{\omega,n},\rho_{\omega,n})\|_{B^{\delta}_{p,r}\times B^{\delta-1}_{p,r}}~+~\|(u^{\omega,n},\rho^{\omega,n})\|_{B^{\delta}_{p,r}\times B^{\delta-1}_{p,r}} \nonumber \\
&~\leq~2\|(u^0_{\omega,n},\rho^0_{\omega,n})\|_{B^{\delta}_{p,r}\times B^{\delta-1}_{p,r}}~+~\|(u^{\omega,n},\rho^{\omega,n})\|_{B^{\delta}_{p,r}\times B^{\delta-1}_{p,r}}~,
\end{align} 
where we have used the solution size estimate in Theorem \ref{existmain}.
Now, for $\delta < s+1$, equations \eqref{initialnorm} and \eqref{approxnorm} imply that $\|(u^0_{\omega,n},\rho^0_{\omega,n})\|_{B^{\delta}_{p,r}\times B^{\delta-1}_{p,r}} \leq C_{\delta}n^{\delta-s}$ and $\|(u^{\omega,n},\rho^{\omega,n})\|_{B^{\delta}_{p,r}\times B^{\delta-1}_{p,r}} \leq C_{\delta}n^{\delta-s}$, where $C_{\delta}$ denotes a constant that depends only on $\delta$, for a given $p$ and $r$. So  \eqref{s2-1} yields
\begin{equation}\label{s2}
\|\mathcal{E}(t)\|_{B^{\delta}_{p,r}\times B^{\delta-1}_{p,r}}~\lesssim ~n^{\delta-s}~.
\end{equation}
We use the interpolation property, (3) from Lemma \ref{besov}, with $\theta=\frac{\delta - s}{\delta - \gamma}$, to obtain
\begin{equation}\label{interpol}
\|\mathcal{E}(t)\|_{B^{s}_{p,r}\times B^{s-1}_{p,r}}~ \leq~\|\mathcal{E}(t)\|^{\theta}_{B^{\gamma}_{p,r}\times B^{\gamma-1}_{p,r}}\|\mathcal{E}(t)\|^{1-\theta}_{B^{\delta}_{p,r}\times B^{\delta-1}_{p,r}}~~.
\end{equation}
From \eqref{interpol}, using \eqref{s1} and \eqref{s2} we get 
\begin{equation}\label{theend}
\|\mathcal{E}(t)\|_{B^{s}_{p,r}\times B^{s-1}_{p,r}} ~\lesssim~ \left(n^{\gamma-s-1}\right)^{\frac{\delta - s}{\delta - \gamma}} \left(n^{\delta-s}\right)^{\frac{s-\gamma}{\delta-\gamma}}~=~n^{-\theta}~.
\end{equation}
As $\theta \in (0,1)$, \eqref{theend} implies that $\|\mathcal{E}(t)\|_{B^{s}_{p,r}\times B^{s-1}_{p,r}} \to 0$ as $n \to \infty$ for any $s> \max\{2+\frac{1}{p}, \frac{5}{2}\}$. This completes the proof of Lemma \ref{lemma}.\\

\end{proof}

\noindent When $p=r=2$, $B^s_{2,2}$ and $H^s$ are equivalent by \cite[Proposition 1.2]{D1} and so we get the following corollary.

\begin{corollary}
The data-to-solution map for the two-component FW system \eqref{FW} is not uniformly continuous from any bounded subset of $H^s(\mathbb{T})\times H^{s-1}(\mathbb{T})$ into $\mathcal{C}([0,T]; H^s(\mathbb{T}))\times\mathcal{C}([0,T]; H^{s-1}(\mathbb{T}))$ for $s>\frac{5}{2}$.
\end{corollary}

\section*{Appendix}
In this appendix, we provide a lower bound on $\|\cos (nx)\|_{B^{s}_{p,r}}$ for any $s>0$ and $1\leq r < \infty$. By Definition \ref{besovdefn},
\begin{equation}\label{cosbesov}
\|\cos (nx)\|_{B^{s}_{p,r}} = \left(\sum\limits_{q\geq -1} 2^{s q r}\|\Delta_q \cos nx\|^r_{{\bf L}^p}\right)^{\frac{1}{r}}~.
\end{equation}
As $\|\Delta_q \cos(nx)\|_{{\bf L}^p} = \varphi_q(n)$ by  Definition \ref{defnbesov}, where $0 < \varphi_q(n) \leq 1$ for all $q$ such that $\frac{1}{\ln (2)}\ln \left(\frac{3}{8}n\right)\leq q \leq \frac{1}{\ln (2)}\ln \left(\frac{4}{3}n\right)$ and $\varphi_q\left(n\right)= 0$ otherwise, \eqref{cosbesov} implies that
\begin{equation*}
\|\cos (nx)\|_{B^{s}_{p,r}}~\geq~ \left(\sum\limits_{q= \frac{1}{\ln (2)}\ln \left(\frac{3}{8}n\right)}^{ \frac{1}{\ln (2)}\ln \left(\frac{4}{3}n\right)} \left(2^q\right)^{sr}\varphi_q^r(n)\right)^{\frac{1}{r}}~.
\end{equation*}
Since $2^q \geq \frac{3}{8}n$ for all terms in the summation, from the above we have
\begin{equation}\label{app1}
\|\cos (nx)\|_{B^{s}_{p,r}}~\geq~ \left(\frac{3}{8}\right)^{s}n^s\left(\sum\limits_{q= \frac{1}{\ln (2)}\ln \left(\frac{3}{8}n\right)}^{ \frac{1}{\ln (2)}\ln \left(\frac{4}{3}n\right)} \varphi_q^r(n)\right)^{\frac{1}{r}} ~.
\end{equation}
Recall that $\varphi_0(\xi) =  \chi\left(\frac{\xi}{2}\right) - \chi(\xi)$ and $\varphi_q(\xi) = \varphi_0(2^{-q}\xi)$ for any $q>-1$, where $\mathrm{supp}~ \chi = [-\frac{4}{3}, \frac{4}{3}]$ and $\chi = 1$ on $[-\frac{3}{4}, \frac{3}{4}]$. This means that $\mathrm{supp}~\varphi_q = [\frac{3}{4}\cdot 2^q, \frac{8}{3}\cdot 2^q]$ for any $q\geq 1$ and furthermore, $\varphi_q = 1$ on the interval $[\frac{4}{3}\cdot 2^q, \frac{3}{2}\cdot 2^q]$. In other words, $\varphi_q(n) = 1$ for $\frac{1}{\ln (2)}\ln \left(\frac{2}{3}n\right)\leq q \leq \frac{1}{\ln (2)}\ln \left(\frac{3}{4}n\right)$. Therefore, from \eqref{app1} we have
\begin{align*}\label{app2}
\|\cos (nx)\|_{B^{s}_{p,r}}~& \geq~ \left(\frac{3}{8}\right)^{s}n^s\left(\sum\limits_{q= \frac{1}{\ln (2)}\ln \left(\frac{2}{3}n\right)}^{ \frac{1}{\ln (2)}\ln \left(\frac{3}{4}n\right)} 1\right)^{\frac{1}{r}} \nonumber \\
& = \left(\frac{3}{8}\right)^{s}n^s\left(\frac{1}{\ln (2)}\left[\ln \left(\frac{3}{4}n\right) - \ln \left(\frac{2}{3}n\right)\right]\right)^{\frac{1}{r}} \nonumber \\
& = \left(\frac{1}{\ln (2)} \ln \left(\frac{9}{8}\right) \right)^{\frac{1}{r}}\left(\frac{3}{8}\right)^{s}n^s~=~ C_s n^s,
\end{align*}
where $C_s$ is a constant that depends only on $s$, for a given $r$. The same estimate holds for $\|\sin (nx)\|_{B^{s}_{p,r}}$ as well.



\end{document}